\documentclass[11pt]{article}

\usepackage{amsmath,amssymb,amsthm,amscd,eucal}
\usepackage{latexsym}
\usepackage{graphicx}
\usepackage[usenames,dvipsnames,svgnames,table]{xcolor}
\usepackage{tikz}
\usepackage{color}

\newtheorem{thm}{Theorem}[section]

\newtheorem{lemma}[thm]{Lemma}
\newtheorem{prop}[thm]{Proposition}

\newtheorem*{conj*}{Conjecture}

\theoremstyle{definition}

\newtheorem{remark}[thm]{Remark}

\newtheorem{exam}[thm]{Example}

\numberwithin{equation}{section}

\renewcommand{\gcd}{\mathsf{gcd}}

\newcommand{\FF}{\mathbb{F}}  
  
\newcommand{\ZZ}{\mathbb{Z}}

\newcommand{\A}{\mathsf{A}}

\newcommand{\V}[2]{\mathsf{V\,}^{#1}_{#2}} 
\newcommand{\W}[2]{\mathsf{W}^{#1}_{g_{#2}}} 

\newcommand{\s}[1]{\mathsf{s}_{#1}} 

\newcommand{\qp}{\FF_q[x, y]}

\def\cent1{\mathsf{Z}(\A_1)}

\def\modd{\, \mathsf{mod} \,}

\usepackage[letterpaper,margin=1.38in]{geometry}

\usepackage{fancyhdr}
\usepackage[yyyymmdd,hhmmss]{datetime}

\newcounter{marg}[section]

\begin{document}  
\title{A multiparameter family of non-weight irreducible representations of the quantum plane and of the quantum Weyl algebra}
\author{Samuel A.\ Lopes\thanks{The author was partially funded by the European Regional Development Fund through the programme COMPETE and by the Portuguese Government through the FCT -- Funda\c c\~ao para a Ci\^encia e a Tecnologia under the project PEst-C/MAT/UI0144/2013.} \ and Jo\~ao N. P. Louren\c co\thanks{The author was supported by Funda\c c\~ao Calouste Gulbenkian through the undergraduate research programme \emph{Novos Talentos em Matem\'atica}.}}
\date{}
\maketitle

\vspace{-.25 truein}  
\begin{abstract} 
We construct a family of irreducible representations of the quantum plane and of the quantum Weyl algebra over an arbitrary field, assuming the deformation parameter is not a root of unity. We determine when two representations in this family are isomorphic, and when they are weight representations, in the sense of \cite{vB97}.\end{abstract}   


\maketitle

\section{Introduction}

Assume throughout that $\FF$ is a field of arbitrary characteristic, not necessarily algebraically closed, with group of units $\FF^*$. Fix $q\in\FF^*$ with $q\neq 1$. The \emph{quantum plane}  is the unital associative algebra 
\begin{equation}
\qp=\FF\{ x, y\}/(yx-qxy)
\end{equation}
with generators $x$ and $y$ subject to the relation $yx=qxy$.

Consider the operators $\tau_q$ and $\partial_q$ defined on the polynomial algebra $\FF[t]$ by
\begin{equation}
 \tau_q (p)(t)=p(qt), \quad \mbox{and} \quad \partial_q (p)(t)=\frac{p(qt)-p(t)}{qt-t}, \quad \mbox{for $p\in\FF[t]$}.
\end{equation}
Then the assignment $x\mapsto \tau_q$, $y\mapsto \partial_q$ yields a (reducible) representation $\qp\rightarrow \mathrm{End}_\FF (\FF[t])$ of $\qp$, which is faithful if and only if $q$ is not a root of unity. The operators $\tau_q$ and $\partial_q$ are central in the theory of linear $q$-difference equations and $\partial_q$ is also known as the \emph{Jackson derivative}, as it appears in \cite{fJ10}. See e.g.\ \cite{yM88}, \cite[Chap.\ IV]{cK95} and references therein for further details.

The irreducible representations of the quantum plane $\qp$ have been classified in~\cite{vB97} using results from~\cite{BvO97}. Following \cite{vB97} we say that a representation of $\qp$ is a \emph{weight representation} if it is semisimple as a representation of the polynomial subalgebra $\FF[H]$ generated by the element $H=xy$. When $q$ is a root of unity all irreducible representations of $\qp$ are finite-dimensional weight representations, and these are well understood. For example, if $\FF$ is algebraically closed and $q$ is a primitive $n$-th root of unity then the irreducible representations of $\qp$ are either $1$ or $n$ dimensional. When $q$ is not a root of unity there are irreducible representations of $\qp$ that are not weight representations, and in particular are not finite dimensional. These turn out to be the \emph{$\FF[H]$-torsionfree} irreducible representations of $\qp$, as they remain irreducible (i.e. nonzero) upon localizing at the nonzero elements of $\FF[H]$. In~\cite[Cor.\ 3.3]{vB97} the torsionfree representations of $\qp$ are classified in terms of elements satisfying certain conditions, but no explicit construction of these representations is given.

We assume $q$ is not a root of unity, and we give an explicit construction of a 3-parameter family $\V{m, n}{f}$ of infinite-dimensional representations of $\qp$ having the following properties (compare Propositions \ref{P:isoclass}, \ref{P:dec} and \ref{P:weight}):
\begin{itemize}
\item $m$ and $n$ are positive integers, and $f:\mathbb{Z} \rightarrow \FF^{*}$ satisfies condition \eqref{prop} below, which essentially encodes $n$ independent parameters from $\FF^*$;
\item $\V{m, n}{f}$ is irreducible if and only if $\gcd(m, n)=1$;
\item if $(m, n)\neq (m', n')$ then $\V{m, n}{f}$ and $\V{m', n'}{f'}$ are not isomorphic;
\item $\V{m, n}{f}$ is a weight representation if and only if $m=n$;
\item if $\FF$ is algebraically closed and $V$ is an irreducible weight representation of $\qp$ that is infinite dimensional, then $V\simeq\V{1, 1}{f}$ for some $f:\mathbb{Z} \rightarrow \FF^{*}$.
\end{itemize}
Thus, in some sense weight and non-weight representations of $\qp$ are rejoined in the family $\V{m, n}{f}$.

The localization of $\qp$ at the multiplicative set generated by $x$ contains a copy of the \emph{$q$-Weyl algebra}, which is the algebra 
\begin{equation}\label{E:qwa}
\mathbb{A}_1(q)=\FF\{ X, Y\}/(YX-qXY-1)
\end{equation}
with generators $X$ and $Y$ subject to the relation $YX-qXY=1$ (see~\eqref{E:qwainqp} for details about this embedding). This is used in Subsection~\ref{SS:qwa} to regard the representations $\V{m, n}{f}$ as infinite-dimensional irreducible representations of $\mathbb{A}_1(q)$. In contrast with the action of $\qp$ on $\V{m, n}{f}$ when $m=n$, it turns out that $\V{m, n}{f}$ is never a weight representation of $\mathbb{A}_1(q)$ in the sense of~\cite{vB97}. In Subsection~\ref{SS:restriction} we pursue a dual approach by constructing representations $\W{n}{}$ of $\mathbb{A}_1(q)$ and then restricting the action from the $q$-Weyl algebra to two distinct subalgebras of $\mathbb{A}_1(q)$ isomorphic to $\qp$.

\section{A family $\V{m, n}{f}$ of infinite-dimensional irreducible representations of $\qp$ for $q$ not a root of unity}\label{S:qnru}

Assume $q\in\FF^*$ is not a root of unity. We introduce a family $\V{m, n}{f}$ of infinite-dimensional representations of $\qp$ which are not in general weight representations in the sense of~\cite{vB97}, but which includes all irreducible infinite-dimensional weight representations of $\qp$ if we further assume $\FF$ to be algebraically closed.

\subsection{Structure of the representations $\V{m, n}{f}$}
 
Fix positive integers $m, n\in\ZZ_{>0}$ and a function $f:\mathbb{Z} \rightarrow \FF^{*}$ satisfying 
\begin{equation}\label{prop}
f(i+n)=qf(i), \quad \quad \text{for all $i\in\ZZ$.}
\end{equation}
Such functions are in one-to-one correspondence with elements of $\left(\FF^*\right)^n$. Let $\V{m, n}{f}$ denote the representation of $\qp$ on the space $\FF[t^{\pm 1}]$ of Laurent polynomials in $t$ given by 
\begin{equation}\label{action}
x.t^{i}=t^{i+n}, \quad \quad y.t^{i}=f(i)t^{i-m},\quad \quad \text{for all $i\in\ZZ$.}
\end{equation}
Condition~\eqref{prop} ensures that the expressions~\eqref{action} do define an action of $\qp$ on $\FF[t^{\pm 1}]$ as, for all $i\in\ZZ$,
\begin{equation*}
(yx-qxy).t^{i}=(f(i+n)-qf(i))t^{i+n-m}=0.
\end{equation*}

\begin{exam}\label{Ex:floor}
Fix $\mu\in\FF^{*}$ and $m, n\in\ZZ_{>0}$. For $i\in\ZZ$ let $f(i)=\mu q^{\left\lfloor \frac{i}{n}\right\rfloor}$, where $\left\lfloor \frac{i}{n}\right\rfloor$ denotes the largest integer not exceeding $\frac{i}{n}$. Then $f:\mathbb{Z} \rightarrow \FF^{*}$ satisfies condition~\eqref{prop} and thus there is a representation $\V{m, n}{f}$ of $\qp$ on $\FF[t^{\pm 1}]$ with action 
\begin{equation*}
x.t^{i}=t^{i+n}, \quad \quad y.t^{i}=\mu q^{\left\lfloor \frac{i}{n}\right\rfloor}t^{i-m},\quad \quad \text{for all $i\in\ZZ$.}
\end{equation*}
\end{exam}


We begin the study of the representations $\V{m, n}{f}$ by first considering the case that the parameters $m$ and $n$ are coprime. The following consequence of~\eqref{prop} will be helpful.


\begin{lemma}\label{L:NT}
Assume  $\gcd(m, n)=1$ and $f:\mathbb{Z} \rightarrow \FF^{*}$ satisfies~\eqref{prop}. For $k\in\ZZ$ define 
\begin{equation}
\s{f}(k)=\prod_{i=0}^{n-1} f(k-im).
\end{equation}
Then $\s{f}(k)=\s{f}(0)q^{k}$.
\end{lemma}

\begin{proof}
For $j\in\ZZ$ let $0\leq \overline{\jmath}<n$ be the unique integer such that  $\overline{\jmath}\equiv j \modd n$. Then the formula $f(j)=f(\overline{\jmath}) q^{\frac{j-\overline{\jmath}}{n}}$ can be verified by induction on $\left| \frac{j-\overline{\jmath}}{n} \right|$. Thus,
\begin{equation*}
 \s{f}(k)=\prod_{i=0}^{n-1} f(k-im)=\prod_{i=0}^{n-1} f\left(\overline{k-\imath m}\right) \prod_{i=0}^{n-1}q^{\frac{k-im-\overline{k-\imath m}}{n}}. 
\end{equation*}
Since $m$ and $n$ are coprime,  the set $\left\{ \overline{k-\imath m} \mid 0\leq i<n \right\}$  consists of all the integers from $0$ to $n-1$, and is thus independent of $k$. Moreover,
\begin{equation*}
\sum_{i=0}^{n-1}{\frac{k-im-\overline{k-\imath m}}{n}} = k+ \sum_{i=0}^{n-1}{\frac{-im-\overline{k-\imath m}}{n}} = k+ \sum_{i=0}^{n-1}{\frac{-im-\overline{(-\imath m)}}{n}}.
\end{equation*}
Hence,
\begin{equation*}
 \s{f}(k) =q^{k}\prod_{i=0}^{n-1} f\left(\overline{-\imath m}\right) \prod_{i=0}^{n-1}q^{\frac{-im-\overline{(-\imath m)}}{n}}=q^{k}\s{f}(0).
\end{equation*}
\end{proof}

\begin{prop}\label{P:irred}
 Assume  $\gcd(m, n)=1$ and $f:\mathbb{Z} \rightarrow \FF^{*}$ satisfies~\eqref{prop}. Then the representation $\V{m, n}{f}$ defined by~\eqref{action} is an irreducible representation of $\qp$.
\end{prop}

\begin{proof}
We begin with a computation: for $k\in\ZZ$ we have, by Lemma~\ref{L:NT},
\begin{equation}\label{E:eigen}
x^m y^n . t^k=x^m\left(\prod_{i=0}^{n-1}f(k-im)\right) t^{k-nm}=\s{f}(k) t^{k}=\s{f}(0)q^{k} t^{k}.
\end{equation}
Hence,  $x^m y^n . p(t)=\s{f}(0)p(qt)$ for all $p\in\FF[t^{\pm 1}]$.

Let $\mathsf{W}\subseteq \V{m, n}{f}$ be a nonzero subrepresentation. If $p(t)\in\mathsf{W}$ then also $p(qt)\in\mathsf{W}$, by~\eqref{E:eigen}. As $q$ is not a root of unity, the latter implies that $t^{\ell}\in\mathsf{W}$ for some $\ell\in\ZZ$. The coprimeness of $m$ and $n$ shows the existence of integers $a$ and $b$ so that $an-bm=1$. By replacing $a$ and $b$ with $a+jm$ and $b+jn$ for a sufficiently large integer $j$, we can assume $a, b\in\ZZ_{>0}$. Then $x^{a} y^{b}.t^{k}=\lambda_{k}t^{k+1}$ for some $\lambda_{k}\in\FF^{*}$, showing that $t^{k}\in\mathsf{W}$ for all $k\geq \ell$. A similar argument shows that $t^{k}\in\mathsf{W}$ for all $k\leq \ell$. Hence $\mathsf{W}=\V{m, n}{f}$, establishing the irreducibility of $\V{m, n}{f}$.
\end{proof}

Next we describe $\V{m, n}{f}$ in terms of a maximal left ideal of $\qp$. Recall that for a representation $\mathsf{V}$ of $\qp$ and an element $v\in\mathsf{V}$, the annihilator of $v$ in $\qp$ is $\mathsf{ann}_{\qp}(v)=\{ r\in\qp \mid r.v=0 \}$, a left ideal of $\qp$.

\begin{prop}\label{P:isoclass}
Assume  $\gcd(m, n)=1$ and $f:\mathbb{Z} \rightarrow \FF^{*}$ satisfies~\eqref{prop}. 
\begin{enumerate}
\item[\textup{(a)}] For $1\in\V{m, n}{f}$, $\mathsf{ann}_{\qp}(1)=\qp\left(x^{m}y^{n}-\s{f}(0)\right)$ and $$\V{m, n}{f}\simeq \qp/\qp\left(x^{m}y^{n}-\s{f}(0)\right).$$
\item[\textup{(b)}] For positive integers $m', n'$, and $f':\mathbb{Z} \rightarrow \FF^{*}$ satisfying \eqref{prop} (with $n$ replaced by $n'$), 
we have $\V{m, n}{f}\simeq \V{m', n'}{f'}$ if and only if $m=m'$, $n=n'$ and $\s{f'}(0)=q^{k}\s{f}(0)$ for some $k\in\ZZ$.
\end{enumerate}
\end{prop}

\begin{proof}
(a)\ Let $\theta=x^m y^n$. First we show that 
\begin{equation}\label{E:claim_ann}
\mathsf{ann}_{\qp}(1)=\qp\left(\FF[\theta]\cap\mathsf{ann}_{\qp}(1) \right).
\end{equation}
The inclusion $\supseteq$ is clear, so suppose $u\in\mathsf{ann}_{\qp}(1)$. Write $u=\sum_{i\geq 0}\mu_i x^{a_i}y^{b_i}=\sum_{k\in\ZZ}u_k$, where $\displaystyle u_k=\sum_{na_i-mb_i=k}\mu_i x^{a_i}y^{b_i}$. Since $u_k.1$ is in $\FF t^k$, it follows that $u_k\in\mathsf{ann}_{\qp}(1)$ for all $k\in\ZZ$, and it suffices to prove $u_k\in\qp\left(\FF[\theta]\cap\mathsf{ann}_{\qp}(1) \right)$.

If $na_i-mb_i=na_j-mb_j$ then, as $\gcd(m, n)=1$, we deduce that $(a_i, b_i)=(a_j, b_j)+\xi (m, n)$ for some $\xi\in\ZZ$. Thus, by the normality of $x$ and $y$, there are  $a, b\geq 0$ with $na-mb=k$ such that $u_k=x^a y^b w_0$, where $w_0=\sum_{j\geq 0}\nu_j x^{\xi_j m}y^{\xi_j n}\in\FF[\theta]$. Notice that for any $\ell\in\ZZ$, $x^a y^b.t^\ell$ is a nonzero scalar multiple of $t^{\ell+k}$, so $x^a y^b w_0=u_k\in\mathsf{ann}_{\qp}(1)$ implies that $w_0\in\mathsf{ann}_{\qp}(1)$. Hence, $u_k\in\qp\left(\FF[\theta]\cap\mathsf{ann}_{\qp}(1) \right)$ and \eqref{E:claim_ann} is established.

Now \eqref{E:eigen} implies that $\theta-\s{f}(0)\in\FF[\theta]\cap\mathsf{ann}_{\qp}(1)$. Since $\FF[\theta]\left( \theta-\s{f}(0) \right)$ is a maximal ideal of $\FF[\theta]$ it follows that $\FF[\theta]\cap\mathsf{ann}_{\qp}(1)=\FF[\theta]\left( \theta-\s{f}(0) \right)$ and $\mathsf{ann}_{\qp}(1)=\qp\left(\theta-\s{f}(0)\right)$. This proves (a) as $1\in\V{m, n}{f}$ generates $\V{m, n}{f}$.

(b)\ We observe that the arguments above also show that for $t^k\in\V{m, n}{f}$, $\mathsf{ann}_{\qp}(t^k)=\qp\left(\theta-q^k \s{f}(0)\right)$ and 
\begin{equation*}
\V{m, n}{f}\simeq \qp/\qp\left(x^{m}y^{n}-q^k\s{f}(0)\right),
\end{equation*}
for any $k\in\ZZ$. This establishes the \textit{if} part of (b). For the direct implication, suppose $\V{m, n}{f}\simeq \V{m', n'}{f'}$. We have, for $a, b\geq 0$ and $t^k\in\V{m, n}{f}$, 
$$
x^a y^b.t^k=\left( \prod_{i=0}^{b-1}f(k-im)\right)t^{k+na-mb}
$$
and $\prod_{i=0}^{b-1}f(k-im)\neq 0$. This implies that $x^a y^b$ is diagonalizable on $\V{m, n}{f}$ if and only if $na=mb$. As $\gcd(m, n)=1$ this amounts to having $(a, b)=\xi(m, n)$ for some $\xi\geq 0$. 

Since $\V{m, n}{f}\simeq \V{m', n'}{f'}$, then $x^{m'}y^{n'}$ is diagonalizable on $\V{m, n}{f}$ and similarly $x^{m}y^{n}$ is diagonalizable on $\V{m', n'}{f'}$. By the relation above we conclude that $(m, n)=(m', n')$. Moreover, the eigenvalues of $x^{m}y^{n}$ on $\V{m, n}{f}$ are of the form $q^{k}\s{f}(0)$, whereas $\s{f'}(0)$ is an eigenvalue of $x^{m'}y^{n'}=x^{m}y^{n}$ on $\V{m', n'}{f'}$. Hence $\s{f'}(0)=q^{k}\s{f}(0)$ for some $k\in\ZZ$, which concludes the proof.
\end{proof}

\begin{remark}\label{R:f}
By  Proposition~\ref{P:isoclass} above, for $\gcd(m, n)=1$ and $f:\mathbb{Z} \rightarrow \FF^{*}$ satisfying~\eqref{prop}, the isomorphism class of $\V{m, n}{f}$ depends only on $m$, $n$ and $\s{f}(0)\in\FF^*$. 

Fix $\lambda\in\FF^*$. Since $\gcd(m, n)=1$ there is a unique $f_\lambda:\mathbb{Z} \rightarrow \FF^{*}$ such that \eqref{prop} holds and $f_\lambda(km)=\lambda$ if $k=0$ and $f_\lambda(km)=1$ if $-(n-1)\leq k\leq -1$. Then $\s{f_\lambda}(0)=\lambda$, $\V{m, n}{f_\lambda}\simeq \qp/\qp\left(x^{m}y^{n}-\lambda\right)$ and, for $\lambda'\in\FF^*$, $\V{m, n}{f_\lambda}\simeq \V{m, n}{f_{\lambda'}}$ if and only if $\lambda/\lambda'\in \langle q\rangle$, where $\langle q\rangle$ is the  subgroup of $\FF^*$ generated by $q$.

If $\FF$ contains an $n$-th root of $\lambda$, say $\mu$, there is a more natural construction for the irreducible representation $\qp/\qp\left(x^{m}y^{n}-\lambda\right)$. Define $f^{\mu}(i)=\mu q^{\left\lfloor \frac{i}{n}\right\rfloor}$, as in Example~\ref{Ex:floor}. Then $\s{f^\mu}(0)=q^k \mu^n=q^k \lambda$, for some $k\in\ZZ$. It follows from Proposition~\ref{P:isoclass} that $\V{m, n}{f^\mu}\simeq \qp/\qp\left(x^{m}y^{n}-\lambda\right)$ and $\V{m, n}{f^\mu}$ depends only on $m$, $n$ and $\lambda$, and not on the particular $n$-th root of $\lambda$ that was chosen.
\end{remark}

Finally we consider the general case of arbitrary $m, n\in\ZZ_{>0}$.

\begin{prop}\label{P:dec}
Let  $m, n\in\ZZ_{>0}$ be arbitrary, with $d=\gcd(m, n)$, and assume $f:\mathbb{Z} \rightarrow \FF^{*}$ satisfies~\eqref{prop}. Then there is a direct sum decomposition
\begin{equation}\label{dec}
\V{m, n}{f}\simeq \bigoplus_{k=0}^{d-1}\V{m/d, n/d}{f_k}
\end{equation}
into irreducible representations, where $f_k(i)=f(k+id)$, for $0\leq k<d$ and $i\in\ZZ$.

Moreover, suppose $m', n'\in\ZZ_{>0}$, and $f':\mathbb{Z} \rightarrow \FF^{*}$ satisfies~\eqref{prop} (with $n$ replaced by $n'$). If  $\V{m, n}{f}\simeq \V{m', n'}{f'}$ then $m=m'$ and $n=n'$.
\end{prop}

\begin{proof}
For $0\leq k<d$,  the subspace $t^k \FF[t^{\pm d}]$ of $\V{m, n}{f}$ is readily seen to be invariant under the actions of $x$ and $y$, and we have $\V{m, n}{f} = \bigoplus_{k=0}^{d-1}t^k \FF[t^{\pm d}]$. Thus, next we argue that the subrepresentation $t^k \FF[t^{\pm d}]$ is isomorphic to $\V{m/d, n/d}{f_k}$, where $f_k(i)=f(k+id)$ for all $i\in\ZZ$. First notice that $f_k(i+n/d)=f(k+id+n)=qf(k+id)=qf_k(i)$, so $\V{m/d, n/d}{f_k}$ is defined. Consider the map $\phi: \V{m/d, n/d}{f_k}\rightarrow t^k \FF[t^{\pm d}]$ given by $\phi(p)(t)=t^kp(t^d)$, for all $p\in\FF[t^{\pm 1}]$. In particular, $\phi(t^i)=t^{k+id}$ for $i\in\ZZ$. Still viewing $t^k \FF[t^{\pm d}]$ as a subrepresentation of $\V{m, n}{f}$, we have:
\begin{align*}
 \phi(x.t^i) &=\phi(t^{i+n/d})=t^{k+id+n}=x.t^{k+id}=x.\phi(t^i),\\
 \phi(y.t^i) &=\phi(f_k(i)t^{i-m/d})=f(k+id)t^{k+id-m}=y.t^{k+id}=y.\phi(t^i).
\end{align*}
Since $\phi$ is clearly bijective, the calculations above show that $\phi$ is an isomorphism of representations, and $\V{m, n}{f}\simeq \bigoplus_{k=0}^{d-1}\V{m/d, n/d}{f_k}$. The fact that each summand $\V{m/d, n/d}{f_k}$ is irreducible follows from $\gcd(m/d, n/d)=1$ and Proposition~\ref{P:irred}, which will be established independently. 

Finally, assume $\V{m, n}{f}\simeq \V{m', n'}{f'}$ for positive integers $m'$ and $n'$, and $f':\mathbb{Z} \rightarrow \FF^{*}$ satisfying $f'(i+n')=qf'(i)$, for all $i\in\ZZ$. Then, up to isomorphism, $\V{m, n}{f}$ and $\V{m', n'}{f'}$ have the same composition factors, and in particular the same composition length. This proves that $d=\gcd(m, n)=\gcd(m', n')$ and that $\V{m/d, n/d}{f_0}\simeq \V{m'/d, n'/d}{f'_k}$ for some $k$. By Proposition~\ref{P:isoclass}, which will also be established independently, we have $m/d=m'/d$ and $n/d=n'/d$, so $m=m'$ and $n=n'$.

\end{proof}

\subsection{Weight representations of the form $\V{m, n}{f}$}\label{SS:weight}

Let us now determine when $\V{m, n}{f}$ is a weight representation in the sense of~\cite{vB97}. Recall that this occurs when $\V{m, n}{f}$ is semisimple as a representation over the polynomial subalgebra $\FF[H]$, where $H=xy$. Assume first that $m=n=1$ and fix $\lambda\in\FF^*$. The map $f_\lambda$ defined in Remark~\ref{R:f} is given by $f_\lambda (i)=\lambda q^{i}$ for all $i\in\ZZ$, and the corresponding representation $\V{1, 1}{f_\lambda}\simeq \qp/\qp\left(H-\lambda\right)$ is irreducible. Since $H.t^i=xy.t^i=\lambda q^i t^i$ for all $i$, the decomposition $\V{1, 1}{f_\lambda}=\bigoplus_{i\in\ZZ}\FF\, t^i$ shows that $\V{1, 1}{f_\lambda}$ is semisimple over $\FF[H]$. Moreover, for $\nu\in\FF^*$, $\V{1, 1}{f_\lambda}\simeq \V{1, 1}{f_\nu}$ if and only if $\lambda/\nu\in\langle q\rangle$, the multiplicative subgroup of $\FF^*$ generated by $q$, by Proposition~\ref{P:isoclass}. In case $\FF$ is algebraically closed, these are all the infinite-dimensional irreducible weight representations of $\qp$, by \cite[Cor. 3.2]{vB97}. Combined with Proposition~\ref{P:isoclass}(b) the above yields the classification of irreducible weight representations in the family $\V{m, n}{f}$.

\begin{prop}\label{P:weight}
Assume  $\gcd(m, n)=1$ and $f:\mathbb{Z} \rightarrow \FF^{*}$ satisfies~\eqref{prop}. Then $\V{m, n}{f}$ is a weight representation if and only if $m=n=1$.
\end{prop}

For completeness, we include a brief and direct proof of Proposition~\ref{P:weight} not assuming that $\FF$ is algebraically closed, a condition that was used implicitly at the end of the previous paragraph.

\begin{proof}
Assume first that $m=n=1$. Then since $f$ satisfies~\eqref{prop} we have $f=f_\lambda$ for $\lambda=f(0)$ and the discussion above shows that $\V{m, n}{f}$ is a weight representation of $\qp$. Conversely, suppose $\V{m, n}{f}$ is a weight representation of $\qp$. Then clearly  $\dim_\FF \FF[H].v<+\infty$ for any $v\in\V{m, n}{f}$. Notice that, for all $i\in\ZZ$, $H.t^i=xy.t^i=f(i)t^{i+n-m}$. Thus, for $\ell\in\ZZ$, $H^\ell.t^i=\zeta t^{i+\ell(n-m)}$ for some $\zeta\in\FF^*$. But then the condition $\dim_\FF \FF[H].1<+\infty$ immediately implies $m=n$, and hence $m=n=1$, as $\gcd(m, n)=1$.
\end{proof}

\begin{remark}
Given arbitrary positive integers $m$ and $n$, and  $f$ satisfying~\eqref{prop}, the representation $\V{m, n}{f}$ is a weight representation if and only if $m=n$. The direct implication follows from the proof of Proposition~\ref{P:weight}. For the converse implication, recall that $\V{m, m}{f}$ is the direct sum of $m$ representations of the form $\V{1, 1}{f_k}$, for $0\leq k<m$, by Proposition~\ref{P:dec}, so the claim follows as each of these is a weight representation.
\end{remark}

\section{Connections with the representation theory of the $q$-Weyl algebra $\mathbb{A}_1(q)$}\label{S:conn}

We continue to assume $q\in\FF^*$ is not a root of unity. Let $\mathbb{A}_1(q)$ be the $q$-Weyl algebra given by generators $X$ and $Y$ and defining relation $YX-qXY=1$, as in~\eqref{E:qwa}. It is straightforward to show that $\{x^k \mid k\geq 0\}$ is a right and left Ore set consisting of regular elements of $\qp$, and we denote the corresponding localization by $\FF_q[x^{\pm 1}, y]$. The calculation
\begin{equation*}
\big(x^{-1}(y-1)\big)x-qx\big(x^{-1}(y-1)\big)=x^{-1}yx -q(y-1)-1= qy -q(y-1)-1=q-1
\end{equation*}
shows that there is an algebra map 
\begin{equation}\label{E:qwainqp}
\mathbb{A}_1(q)\rightarrow \FF_q[x^{\pm 1}, y], \quad \mbox{with \quad $X\mapsto x,\quad Y\mapsto \frac{1}{q-1}x^{-1}(y-1)$.}
\end{equation}
To see that the map in \eqref{E:qwainqp} is injective we can argue as follows. The multiplicative subset $\{X^k \mid k\geq 0\}$ of $\mathbb{A}_1(q)$ is a right and left Ore set of regular elements and we denote the corresponding localization by $\widehat{\mathbb{A}}_1(q)$. Then the map in \eqref{E:qwainqp} extends to a map $\widehat{\mathbb{A}}_1(q)\rightarrow \FF_q[x^{\pm 1}, y]$, which has an inverse $\FF_q[x^{\pm 1}, y] \rightarrow \widehat{\mathbb{A}}_1(q)$ with $x^{\pm 1}\mapsto X^{\pm 1}$ and $y\mapsto (q-1)XY+1$. It follows that \eqref{E:qwainqp} induces an isomorphism $\widehat{\mathbb{A}}_1(q)\simeq \FF_q[x^{\pm 1}, y]$, and in particular \eqref{E:qwainqp} is injective. In view of the above we will identify $X$ with $x$, $Y$ with $\frac{1}{q-1}x^{-1}(y-1)$ and $\mathbb{A}_1(q)$ with the corresponding subalgebra of $\FF_q[x^{\pm 1}, y]$. Since $y=(q-1)XY+1=YX-XY$, we have the embeddings
\begin{equation}\label{E:embed}
\qp \subseteq \mathbb{A}_1(q)\subseteq \FF_q[x^{\pm 1}, y]=\widehat{\mathbb{A}}_1(q).
\end{equation}

\subsection{Extension of the representations $\V{m, n}{f}$ to $\mathbb{A}_1(q)$}\label{SS:qwa}

Our aim in this subsection is to extend the action of $\qp$ on $\V{m, n}{f}$ to an action of the $q$-Weyl algebra $\mathbb{A}_1(q)$. Assume thus that $m, n$ are positive integers and $f:\mathbb{Z} \rightarrow \FF^{*}$ satisfies~\eqref{prop}. If $\rho^{m, n}_{f}:\qp\rightarrow \mathrm{End}_\FF (\V{m, n}{f})$ is the representation of $\qp$ on $\V{m, n}{f}$, we first observe that $\rho^{m, n}_{f}(x)$ is an invertible linear map on $\V{m, n}{f}$, a fact which is clear from \eqref{action}. Therefore $\rho^{m, n}_{f}$ extends to the localization $\FF_q[x^{\pm 1}, y]$, and  $\V{m, n}{f}$ can be seen as a representation of $\FF_q[x^{\pm 1}, y]$ with $x^{-1}.t^{i}=t^{i-n}$ for all $i\in\ZZ$. Now we get an action of $\mathbb{A}_1(q)$ on $\V{m, n}{f}=\FF[t^{\pm 1}]$ by restricting $\rho^{m, n}_{f}$:
\begin{align}\label{E:qwaaction}\nonumber
&X.t^{i} =x.t^{i}=t^{i+n},\\  
&Y.t^{i} =\frac{1}{q-1}x^{-1}(y-1).t^{i}=\frac{1}{q-1}(f(i)t^{i-m-n}-t^{i-n}),\quad \quad \text{for all $i\in\ZZ$.}
\end{align}

In our next result we view $\V{m, n}{f}$ as a representation of $\mathbb{A}_1(q)$, as above.

\begin{prop}
 Assume  $\gcd(m, n)=1$ and $f:\mathbb{Z} \rightarrow \FF^{*}$ satisfies~\eqref{prop}. Then:
 \begin{enumerate}
 \item[\textup{(a)}] $\V{m, n}{f}$ defined by~\eqref{E:qwaaction} is an irreducible representation of $\mathbb{A}_1(q)$.
\item[\textup{(b)}]  For positive integers $m', n'$, and $f':\mathbb{Z} \rightarrow \FF^{*}$ satisfying \eqref{prop} (with $n$ replaced by $n'$), 
we have $\V{m, n}{f}\simeq \V{m', n'}{f'}$ as representations of $\mathbb{A}_1(q)$ if and only if $m=m'$, $n=n'$ and $\s{f'}(0)=q^{k}\s{f}(0)$ for some $k\in\ZZ$.
\item[\textup{(c)}] $\V{m, n}{f}$ is not semisimple as a representation over the polynomial subalgebra of $\mathbb{A}_1(q)$ generated by $XY$; hence, $\V{m, n}{f}$ is not a weight representation of $\mathbb{A}_1(q)$ in the sense of \cite{vB97}.
\end{enumerate}
\end{prop}
\begin{proof}
Part (a) and the direct implication in (b) follow from the embedding \eqref{E:embed}, and from Propositions \ref{P:irred} and \ref{P:isoclass}.
 
Suppose now $f':\mathbb{Z} \rightarrow \FF^{*}$ satisfies \eqref{prop}, and there is $k\in\ZZ$ so that $\s{f'}(0)=q^{k}\s{f}(0)$. Then by  Proposition~\ref{P:isoclass} there is an isomorphism $\phi : \V{m, n}{f}\rightarrow \V{m, n}{f'}$ as representations of $\qp$. For $v\in\V{m, n}{f}$ we have $\phi(v)=\phi(xx^{-1}.v)=x.\phi(x^{-1}.v)$, thus $\phi(x^{-1}.v)=x^{-1}.\phi(v)$. Whence $\phi$ is an isomorphism of representations of $\FF_q[x^{\pm 1}, y]$. The other implication in (b) now follows from \eqref{E:embed}.

Observe that $XY=\frac{1}{q-1}(y-1)$, so the polynomial subalgebra of $\mathbb{A}_1(q)$ generated by $XY$ is just $\FF[y]$. Given $0\neq v\in\V{m, n}{f}$, the formula $y.t^{i}=f(i)t^{i-m}$ for $i\in\ZZ$ implies $\dim_\FF \FF[y].v=+\infty$. Hence, $\V{m, n}{f}$ is not semisimple over $\FF[y]=\FF[XY]$, and therefore it is not a weight representation of $\mathbb{A}_1(q)$ in the sense of \cite{vB97}.
\end{proof}

\begin{remark}
In~\cite{BO09} the authors introduce Whittaker representations for generalized Weyl algebras. For the cases covered in this note, a representation $\mathsf{V}$ is a \emph{Whittaker representation} for $\qp$ (respectively, for $\mathbb{A}_1(q)$) if $\mathsf{V}$ is generated by an element $v\in \mathsf{V}$ which is an eigenvector for the action of $x\in\qp$ (respectively, for the action of $X\in \mathbb{A}_1(q)$). Since $m, n\geq 1$, it is immediate that the operators $x, y\in\qp$ (respectively, $X, Y\in\mathbb{A}_1(q)$) have no eigenvectors in $\V{m, n}{f}$, so $\V{m, n}{f}$ is not a Whittaker representation for the quantum plane (respectively, for the $q$-Weyl algebra).
\end{remark}

\subsection{The representations $\W{n}{}$ of $\mathbb{A}_1(q)$ and their restriction to $\qp$}\label{SS:restriction}

We will now use a similar idea to construct representations of the $q$-Weyl algebra on the Laurent polynomial algebra $\FF[t^{\pm 1}]$. Fix positive integers $m, n\in\ZZ_{>0}$ and a function $g:\mathbb{Z} \rightarrow \FF$. Then the formulas 
\begin{equation}\label{qwa:action}
X.t^{i}=t^{i+n}, \quad \quad Y.t^{i}=g(i)t^{i-m},\quad \quad \text{for all $i\in\ZZ$}
\end{equation}
yield a representation of $\mathbb{A}_1(q)$ on $\FF[t^{\pm 1}]$ if and only if $m=n$ and $g$ satisfies
\begin{equation}\label{qwa:prop}
g(i+n)=qg(i)+1, \quad \quad \text{for all $i\in\ZZ$.}
\end{equation}
We denote the corresponding representation of $\mathbb{A}_1(q)$ by $\W{n}{}$. Notice that for all $i\in\ZZ$
\begin{equation}\label{qwa:xy:com}
XY.t^i=g(i)t^i, \quad \quad (YX-XY).t^i=\left(g(i+n)-g(i)\right)t^i=\left((q-1)g(i)+1\right)t^i,
\end{equation}
so $\W{n}{}$ is a weight representation of $\mathbb{A}_1(q)$ in the sense of \cite{vB97}.

\begin{remark}
It follows from the computations at the beginning of Section~\ref{S:conn} that the element $YX-XY$ is normal in $\mathbb{A}_1(q)$ and it is sometimes referred to as a Casimir element, in spite of not being central. The equality $YX-XY=(q-1)XY+1$ shows that $YX-XY$ and $(q-1)XY+1$ generate the same subalgebra of $\mathbb{A}_1(q)$ and thus a weight representation of $\mathbb{A}_1(q)$ could be defined in an equivalent manner as a representation which is semisimple over the subalgebra generated by the Casimir element $YX-XY$. 
\end{remark}

Our first observation is the analogue of Proposition~\ref{P:dec}.

\begin{lemma}\label{L:qwa:dec}
Let  $n\in\ZZ_{>0}$ and assume $g:\mathbb{Z} \rightarrow \FF$ satisfies~\eqref{qwa:prop}. There is a direct sum decomposition
\begin{equation}\label{qwa:dec}
\W{n}{}\simeq \bigoplus_{k=0}^{n-1}\W{1}{k},
\end{equation}
where $g_k(i)=g(k+in)$, for $0\leq k<n$ and $i\in\ZZ$. 
\end{lemma}
\begin{proof}
For  $0\leq k<n$, the subspace $t^k \FF[t^{\pm n}]$ is invariant under the actions of $X$ and $Y$ and $\W{n}{}=\bigoplus_{k=0}^{n-1}t^k \FF[t^{\pm n}]$. Moreover, the map $\phi: \W{1}{k}\rightarrow t^k \FF[t^{\pm n}]$ given by $\phi(p)(t)=t^kp(t^n)$, for all $p\in\FF[t^{\pm 1}]$ is easily checked to be an isomorphism.
\end{proof}

In view of the above, it is enough to study the structure of the representations $\W{1}{}$, where $g:\mathbb{Z} \rightarrow \FF$ satisfies $g(i+1)=qg(i)+1$ for all $i\in\ZZ$. Equivalently, $g(i)=g(0)q^i + [i]_q$, where $[i]_q=\frac{q^i-1}{q-1}$ for all $i\in\ZZ$. 

\begin{prop}\label{P:qwa:w:irriso}
Let $g, g':\mathbb{Z} \rightarrow \FF$ satisfy~\eqref{qwa:prop} with $n=1$. Then:
\begin{enumerate}
\item[\textup{(a)}] $\W{1}{}\simeq\mathsf{W}^{1}_{g'}$ if and only if $g(0)=g'(i)$ for some $i\in\ZZ$;
\item[\textup{(b)}] $\W{1}{}$ is irreducible if and only if $g(0)\notin \{ [i]_q \mid i\in\ZZ\}\cup\left\{ -\frac{1}{q-1}\right\}$.
\end{enumerate}
\end{prop}

\begin{proof}
For (a), suppose $\W{1}{}\simeq\mathsf{W}^{1}_{g'}$. By~\eqref{qwa:xy:com} the eigenvalues of $XY$ on $\W{1}{}$ are $g(i)$, for $i\in\ZZ$ and similarly the eigenvalues of $XY$ on $\mathsf{W}^{1}_{g'}$ are $g'(i)$, for $i\in\ZZ$. Thus, $g$ and $g'$ must have the same image and in particular $g(0)=g'(i)$ for some $i\in\ZZ$. Conversely, if the latter holds then the map $\phi:\W{1}{}\rightarrow\mathsf{W}^{1}_{g'}$ given by $\phi(p)(t)=t^ip(t)$ for all $p\in\FF[t^{\pm 1}]$ is an isomorphism.

For (b), first observe that for $i\in\ZZ$ we have $g(0)=[i]_q\iff g(-i)=0$. Thus, if $g(0)=[i]_q$ for some $i\in\ZZ$, then $t^{-i}\FF[t]$ is invariant under the actions of $X$ and $Y$, so $\W{1}{}$ is not irreducible in this case. Next observe that $g(0)=-\frac{1}{q-1}\iff g$ is not injective $\iff g$ is constant. It follows that if $g(0)=-\frac{1}{q-1}$, then $(t-1)\FF[t^{\pm 1}]$ is a proper subrepresentation and hence $\W{1}{}$ is not irreducible. This proves the direct implication in (b). For the converse, by the observations above, we can assume that $g(i)\neq 0$ for all $i\in\ZZ$ and that $g$ is injective. Let $\mathsf{S}$ be a nonzero subrepresentation of $\W{1}{}$. By repeatedly applying the operator $X$ to a chosen nonzero element of $\mathsf{S}$, we will obtain a nonzero element of $\mathsf{S}\cap\FF[t]$. Let $p$ be one such element, chosen so that it has minimum degree, say $p=\sum_{k=0}^d a_k t^k$, with $a_d\neq 0$. Since $g(i)\neq 0$ for all $i\in\ZZ$, the minimality of $p$ implies that $a_0\neq 0$. Then 
\begin{equation*}
\mathsf{S}\cap\FF[t]\ni (XY-g(d)).p=\sum_{k=0}^{d-1} (g(k)-g(d))a_k t^k.
\end{equation*}
By the minimality of $p$ we must have $(XY-g(d)).p=0$. Hence, $g(0)=g(d)$ and the injectivity of $g$ gives $d=0$. It follows that $t^0\in\mathsf{S}$ and thus $\mathsf{S}=\W{1}{}$.
\end{proof}

Now that we understand the representations $\W{n}{}$, we will consider their restriction to $\qp$ via each of the two embeddings
\begin{align}\label{E:emb:sigma}
\sigma &: \qp\rightarrow \mathbb{A}_1(q),  \quad\quad x\mapsto X, \quad y\mapsto YX-XY=(q-1)XY+1;\\ \label{E:emb:tau}
\tau &: \qp\rightarrow \mathbb{A}_1(q),  \quad\quad x\mapsto YX-XY=(q-1)XY+1, \quad y\mapsto Y.
\end{align}

We consider first the restriction relative to $\sigma$. In this case, the action of $\qp$ on $\W{n}{}$ is given by
\begin{equation}\label{rest:w:qp}
x.t^{i}=t^{i+n}, \quad \quad y.t^{i}=\left((q-1)g(i)+1\right)t^i,\quad \quad \text{for all $i\in\ZZ$.}
\end{equation}

\begin{lemma}\label{L:qwa:rest:sigma}
Consider the restriction map $\sigma$ given in~\eqref{E:emb:sigma} to view the representations $\W{n}{}$ as representations of $\qp$.
\begin{enumerate}
\item[\textup{(a)}] Let  $n\in\ZZ_{>0}$ and assume $g:\mathbb{Z} \rightarrow \FF$ satisfies~\eqref{qwa:prop}. Then $\W{n}{}\simeq \bigoplus_{k=0}^{n-1}\W{1}{k}$ as representations of $\qp$, where $g_k(i)=g(k+in)$, for $0\leq k<n$ and $i\in\ZZ$.

\item[\textup{(b)}] Let $g, g':\mathbb{Z} \rightarrow \FF$ satisfy~\eqref{qwa:prop} with $n=1$. Then $\W{1}{}\simeq\mathsf{W}^{1}_{g'}$ as representations of $\qp$ if and only if $g(0)=g'(i)$ for some $i\in\ZZ$.

\item[\textup{(c)}] Assume $g:\mathbb{Z} \rightarrow \FF$ satisfies~\eqref{qwa:prop} with $n=1$. Then $\W{1}{}$ has trivial socle as a representation of $\qp$, i.e., it has no irreducible $\qp$-subrepresentations.
\end{enumerate}
\end{lemma}
\begin{proof}
Part (a) follows directly from Lemma~\ref{L:qwa:dec} and part (b) follows from the proof of Proposition~\ref{P:qwa:w:irriso}(a), as the argument for the direct implication in Proposition~\ref{P:qwa:w:irriso}(a) used only the restriction of the action to the subalgebra generated by $XY$, which coincides with the subalgebra generated by $YX-XY$.

For part (c), suppose by way of contradiction that $\mathsf{S}$ is an irreducible $\qp$-subrepresentation of $\W{1}{}$. Let $0\neq s\in\mathsf{S}$. Then $x.s\neq 0$ and thus $\qp x.s=\mathsf{S}$, which is a contradiction as $s\notin \qp x.s$. 
\end{proof}

\begin{remark}
In the conditions of Lemma~\ref{L:qwa:rest:sigma}, it can be checked that $\W{1}{}$ has maximal $\qp$-subrepresentations if and only if $g$ is constant.
\end{remark}

Now we consider the restriction of $\W{n}{}$ to $\qp$ relative to the map $\tau$ defined in~\eqref{E:emb:tau}. In this case, the action of $\qp$ on $\W{n}{}$ is given by
\begin{equation}\label{rest:w:qp:tau}
x.t^{i}=\left((q-1)g(i)+1\right)t^i, \quad \quad y.t^{i}=g(i)t^{i-n},\quad \quad \text{for all $i\in\ZZ$.}
\end{equation}

\begin{lemma}\label{L:qwa:rest:tau}
Consider the restriction map $\tau$ given in~\eqref{E:emb:tau} to view the representations $\W{n}{}$ as representations of $\qp$.
\begin{enumerate}
\item[\textup{(a)}] Let  $n\in\ZZ_{>0}$ and assume $g:\mathbb{Z} \rightarrow \FF$ satisfies~\eqref{qwa:prop}. Then $\W{n}{}\simeq \bigoplus_{k=0}^{n-1}\W{1}{k}$ as representations of $\qp$, where $g_k(i)=g(k+in)$, for $0\leq k<n$ and $i\in\ZZ$.

\item[\textup{(b)}] Let $g, g':\mathbb{Z} \rightarrow \FF$ satisfy~\eqref{qwa:prop} with $n=1$. Then $\W{1}{}\simeq\mathsf{W}^{1}_{g'}$ as representations of $\qp$ if and only if $g(0)=g'(i)$ for some $i\in\ZZ$.

\item[\textup{(c)}] Assume $g:\mathbb{Z} \rightarrow \FF$ satisfies~\eqref{qwa:prop} with $n=1$. If $g(0)\notin \{ [i]_q \mid i\in\ZZ\}$ then $\W{1}{}$ has trivial socle as a representation of $\qp$, i.e., it has no irreducible $\qp$-subrepresentations. If $g(0)=[i]_q$ for some $i\in\ZZ$ then $\FF t^{-i}$ is the unique irreducible $\qp$-subrepresentation of $\W{1}{}$.
\end{enumerate}
\end{lemma}

\begin{proof}
The proof is the same as the proof of Lemma~\ref{L:qwa:rest:sigma}, except for part (c). For this part, suppose that $\mathsf{S}$ is an irreducible $\qp$-subrepresentation of $\W{1}{}$. If there is $0\neq s\in\mathsf{S}$ such that $y.s\neq 0$, then we obtain a contradiction as in the proof of Lemma~\ref{L:qwa:rest:sigma}(c), showing that no such irreducible $\qp$-subrepresentation of $\W{1}{}$ exists. If $g(0)\notin \{ [i]_q \mid i\in\ZZ\}$ then $g(i)\neq 0$ for all $i\in\ZZ$, so $y.s\neq 0$ for all $s\neq 0$ and the first claim follows. Now suppose $g(0)=[i]_q$ for some $i\in\ZZ$. Then $g(k)=0\iff k=-i$. In particular, $x.t^{-i}=t^{-i}$ and $y.t^{-i}=0$, so that $\FF t^{-i}$ is an irreducible $\qp$-subrepresentation of $\W{1}{}$. If $\mathsf{S}$ is any irreducible $\qp$-subrepresentation of $\W{1}{}$, then the argument above implies that $y.s=0$ for all $s\in\mathsf{S}$, and this in turn implies that $\mathsf{S}\subseteq\FF t^{-i}$, which establishes the second claim in (c).
\end{proof}

\medskip

\begin{flushleft}
\textbf{Acknowledgments.} The authors wish to thank G.~Benkart and M.~Ondrus for helpful comments and suggestions on a preliminary version of this manuscript. They would also like to thank the anonymous referees for their valuable comments and for suggesting the approach in Subsection~\ref{SS:restriction}.
\end{flushleft}



\noindent \textsc{Samuel A. Lopes} \\
\textit{\small CMUP, Faculdade de Ci\^encias, Universidade do Porto, 
Rua do Campo Alegre 687\\ 
4169-007 Porto, Portugal}\\
\texttt{slopes@fc.up.pt}\\

\noindent \textsc{Jo\~ao N. P. Louren\c co} \\
\textit{\small Faculdade de Ci\^encias, Universidade do Porto, 
Rua do Campo Alegre 687\\ 
4169-007 Porto, Portugal}\\
\texttt{jnunolour@gmail.com}

\end{document}